\documentclass[11pt,reqno]{amsart}
\usepackage[foot]{amsaddr}

\usepackage{amssymb,a4wide}

\usepackage{fullpage}
\usepackage{pxfonts}

\usepackage[shortlabels]{enumitem}
\usepackage{xcolor,xspace}

\usepackage{hyperref,doi}

\newcommand{\cal}{\mathcal}
\newcommand{\coeff}[1]{\mathop{[\![#1]\!]}\xspace}

\theoremstyle{thmstyleone}%
\newtheorem{theorem}{Theorem}%  meant for continuous numbers
\newtheorem{lemma}[theorem]{Lemma}% 

\theoremstyle{thmstyletwo}%

\theoremstyle{thmstylethree}%

\title[Sums of half-integer powers of consecutive integers]{Generalization of Ramanujan's formula for sums of half-integer powers of consecutive integers via formal Bernoulli series}
\author{Max A. Alekseyev}
\address[MAA]{The George Washington University, Washington, DC, USA.}
\author{Rafael Gonzalez}
\author{Keryn Loor}
\author{Aviad Susman}
\author{Cesar Valverde}
\address[RG]{Lehman College of the City University of New York, Bronx, NY, USA}
\address[AS]{Icahn School of Medicine at Mount Sinai, New York, NY, USA}
\address[KL, CV]{Medgar Evers College of the City University of New York, Brooklyn NY, USA}


\begin{document}

\maketitle

\begin{abstract}
Faulhaber's formula expresses the sum of the first $n$ positive integers, each raised to an integer power $p\geq 0$, as a polynomial in $n$ of degree $p+1$. Ramanujan expressed this sum for $p\in\{\frac12,\frac32,\frac52,\frac72\}$ as the sum of a polynomial in $\sqrt{n}$ and a certain infinite series. In the present work, we explore the connection to Bernoulli polynomials, and by generalizing those to formal series, we extend the Ramanujan result to all positive half-integers $p$.
\end{abstract}

\section{Introduction}

Let $\mathbb{Z}_{\geq0}$ denote the set of nonnegative integers, and $B_n$ denote the $n$-th Bernoulli number, with $B_1 = -\frac12$. For $n,p \in \mathbb{Z}_{\geq0}$, Faulhaber's formula~\cite[Proposition 9.2.12]{Cohen07v2} states that 
\begin{equation}
\sum_{i=1}^n i^p = \frac1{p+1}\sum_{i=0}^p \binom{p+1}{i} (-1)^i B_i n^{p+1-i}.
\end{equation}
A natural question is whether there exist similar formulae for fractional powers $p$.

The study of this question for half-integer powers was pioneered by Ramanujan~\cite{ram}, who introduced a function
\begin{equation}
\tau(n,m):=\sum_{\nu=0}^{\infty} \dfrac{1}{(\sqrt{n+\nu}+\sqrt{n+\nu+1})^m} \quad 
(n,m\in \mathbb{Z}_{\geq0},\ m\geq 3)
\end{equation}
and proved existence of constants $C_1, C_3, C_5, C_7$ such that
\begin{align}
\sum_{i=1}^ni^{\frac{1}{2}}& = C_1 + \dfrac{2}{3}n^\frac{3}{2}+\dfrac{1}{2}n^\frac{1}{2} + \dfrac{1}{6}\tau(n,3),\label{ram12}\\
\sum_{i=1}^ni^{\frac{3}{2}}& = C_3 + \dfrac{2}{5}n^\frac{5}{2}+\dfrac{1}{2}n^\frac{3}{2} + \dfrac{1}{8}n^\frac{1}{2} + \dfrac{1}{40}\tau(n,5),\label{ram32}\\
\sum_{i=1}^ni^{\frac{5}{2}}& = C_5 + \dfrac{2}{7}n^\frac{7}{2}+\dfrac{1}{2}n^\frac{5}{2} + \dfrac{5}{24}n^\frac{3}{2} - \dfrac{1}{96}\tau(n,3)+\dfrac{1}{224}\tau(n,7),\label{ram52}\\
\sum_{i=1}^ni^{\frac{7}{2}}& = C_7 + \dfrac{2}{9}n^\frac{9}{2}+\dfrac{1}{2}n^\frac{7}{2} + \dfrac{7}{24}n^\frac{5}{2} - \dfrac{7}{384}n^\frac{1}{2}-\dfrac{1}{192}\tau(n,5)+\dfrac{1}{1152}\tau(n,9).\label{ram72}
\end{align}

Following Ramanujan, we reproduce a proof of \eqref{ram12}. We consider the function on $\mathbb{Z}_{\geq0}$ defined by
\begin{equation}\label{phi1}
\phi_1(n):=\sum_{i=1}^n\sqrt{i} - C_1 - \dfrac{2}{3}n^\frac32-\dfrac{1}{2}n^\frac12 - \dfrac{1}{6}\tau(n,3),
\end{equation}
where $C_1$ is defined so that $\phi_1(1) = 0$. We have 
\begin{multline}
\phi_1(n) - \phi_1(n+1) = -\sqrt{n+1}+\dfrac{2}{3}((n+1)^\frac{3}{2}-n^\frac{3}{2})+\dfrac{1}{2}(\sqrt{n+1}-\sqrt{n}) - \dfrac{1}{6}\dfrac{1}{(\sqrt{n}+\sqrt{n+1})^3}\\
= -\sqrt{n+1}+\dfrac{2}{3}((n+1)^\frac{3}{2}-n^\frac{3}{2})+\dfrac{1}{2}(\sqrt{n+1}-\sqrt{n}) + \dfrac{1}{6}(\sqrt{n}-\sqrt{n+1})^3 = 0,   
\end{multline}
 %\phi_1(n) - \phi_1(n+1) = -\sqrt{n+1}+\dfrac{2}{3}((n+1)^\frac{3}{2}-n^\frac{3}{2})+\dfrac{1}{2}(\sqrt{n+1}-\sqrt{n}) - \dfrac{1}{6}\dfrac{1}{(\sqrt{n}+\sqrt{n+1})^3}
 %\end{equation}
 %= -\sqrt{n+1}+\dfrac{2}{3}((n+1)^\frac{3}{2}-n^\frac{3}{2})+\dfrac{1}{2}(\sqrt{n+1}-\sqrt{n}) + \dfrac{1}{6}(\sqrt{n}-\sqrt{n+1})^3 = 0,
so that $\phi_1$ is identically zero. The identities \eqref{ram32}-\eqref{ram72} can be proved in a similar way. Ramanujan remarks that \emph{`the corresponding results for higher powers are not so neat as the previous ones,'} and provides \eqref{ram52} and \eqref{ram72} as examples~\cite[page 174]{ram}. The purpose of this article is to elucidate the general case of Ramanujan's results. 

We remark that there are few results that extend Faulhaber's formula to noninteger powers. For example, \cite{shek} estimates $\sum_{i=1}^ni^{\frac{1}{r}} - \frac{r}{r+1}(n+1)^{\frac{1+r}{r}}+\frac{1}{2}(n+1)^{\frac{1}{r}}$, while \cite[Proposition 9.2.13]{Cohen07v2} and \cite{gown} prove asymptotic formulas for a complex power. Our approach is based on generalization of Bernoulli polynomials to formal power series, establishing certain fundamental properties of the latter. A different generalization of Bernoulli polynomials was proposed in \cite[Section 9]{knuth1993}, resulting in an infinite-series generalization of Faulhaber's formula. In contrast, our result gives an exact and finite formula in the case of a positive half-integer power.

\section{Main Result}

Our main result generalizing Ramanujan's identities \eqref{ram12}-\eqref{ram72} is given by the following theorem.

\begin{theorem}\label{th:main}
For any odd positive integer $k$, we have
\begin{equation}\label{e1}
\sum_{i=1}^n i^\frac{k}{2} = C_k + \sqrt{n}\,P_k(n) + \sum_{i=3\atop i\text{ is odd}}^{k+2} A^k_i\,\tau(n,i),
\end{equation}
where 
\begin{itemize}
    \item 
$P_k(n)$ is a polynomial in $n$:
\begin{equation}\label{Pkn}
P_k(n) :=  \frac1{\frac{k}2+1}\sum_{i=0}^{\frac{k+1}{2}} \binom{\frac{k}{2}+1}{i} (-1)^i B_i n^{\frac{k+1}{2}-i},
\end{equation}
\item $A^k_i$ ($i\in\{1,3,5,\dots,k+2\}$) are constant coefficients given by
\begin{equation}\label{Aki}
A^k_i := \frac{1}{(\frac{k}{2}+1)2^{i-1}}\sum_{j=0}^{\frac{k-i}{2}+1}\binom{2j+i}{j}\frac{1}{4^j}\binom{\frac{k}{2}+1}{\frac{k-i}{2}+1-j}B_{\frac{k-i}{2}+1-j},
\end{equation}
%and 
\item $C_k$ is a constant given by:
\begin{equation}\label{Ck}
C_k := 1 - P_k(1) - \sum_{i=3\atop i\text{ is odd}}^{k+2} A^k_i\,\tau(1,i).
\end{equation}
\end{itemize}
\end{theorem}

To illustrate how Ramanujan's identities \eqref{ram12}-\eqref{ram72} continue, we list
the numerical values of the coefficients $A^k_i$ for $k\in\{9,11,13,15\}$ in 
Table~\ref{tab:Aki}.
The apparent pattern of zeros among the values of $A^k_i$ is established in Theorem~\ref{th:gf_A_2} in the next section. Together with Theorem~\ref{th:gf_A_1} it also provides the generating functions for the coefficients $A^k_i$ and polynomials $P_k(n)$.

\begin{table}[!tb]
\begin{center}
{\renewcommand{\arraystretch}{1.5}%
    \begin{tabular}{r|ccccccccc}
       $k$ & $A^k_1$ & $A^k_3$ & $A^k_5$ & $A^k_7$ & $A^k_9$ & $A^k_{11}$ & $A^k_{13}$ & $A^k_{15}$ & $A^k_{17}$ \\
       \hline\hline
       9  & $0$ & $\frac1{256}$ & $0$ & $-\frac1{512}$ & $0$ & $\frac1{5632}$ &&& \\
       \hline
       $11$  & $0$ & $0$ & $\frac{33}{10240}$ & $0$ & $-\frac1{1536}$ & $0$ & $\frac1{26624}$ && \\
       \hline
       $13$  & $0$ & $-\frac{143}{40960}$ & $0$ & $\frac{221}{122880}$ & $0$ & $-\frac5{24576}$ & $0$ & $\frac1{122880}$ & \\
       \hline
       $15$  & $0$ & $0$ & $-\frac{65}{16384}$ & $0$ & $\frac{41}{49152}$ & $0$ & $-\frac1{16384}$ & $0$ & $\frac1{557056}$ \\
       \hline
    \end{tabular}
}
\end{center}
\caption{Values of $A^k_i$ for $k\in\{9,11,13,15\}$.}
    \label{tab:Aki}
\end{table}

\section{Connection to formal Bernoulli series}

For any real number $\alpha$, we define the \emph{Bernoulli series} ${\cal B}_\alpha(z)$ by
\begin{equation}\label{Bseries}
{\cal B}_\alpha(z) := \sum_{i=0}^\infty \binom{\alpha}{i} B_i z^i.
\end{equation}
When $\alpha=m$ is a nonnegative integer, ${\cal B}_m(z)$ represents a polynomial expressed in terms of the \emph{Bernoulli polynomial} ${\tilde {\cal B}}_m(z) := \sum_{i=0}^m \binom{m}{i} B_i z^{m - i}$ as
\begin{equation}\label{recipBP}
{\cal B}_m(z) = z^m {\tilde {\cal B}}_m(\frac1{z}), \qquad (m\in\mathbb{Z}_{\geq0})
\end{equation}
which is called the \emph{reciprocal Bernoulli polynomial}~\cite{Kellner2023}.

Since Bernoulli numbers with odd indices are zero, except for $B_1=-\frac12$, we easily derive
\begin{equation}\label{eq:B_minus_z}
{\cal B}_\alpha(-z) = {\cal B}_\alpha(z) + \alpha z.
\end{equation}
The Bernoulli series further satisfy the following property:

\begin{theorem}\label{th:B_duality}
For any power series $h(z)$, real number $\alpha$, and integer $m$,
$$\coeff{z^m} h(z) {\cal B}_\alpha(z) = \coeff{z^m} h(\frac{z}{1+z})(1+z)^{m-1-\alpha} {\cal B}_\alpha(-z),$$
where $\coeff{z^m}$ denotes the operator producing the coefficient of $z^m$.
\end{theorem}

\begin{proof}
Since the identity is formed by polynomials in $\alpha$ (of degree at most $m$), it is sufficient to show that it holds for any positive integer $\alpha$.

For a positive integer $\alpha$, the representation~\eqref{recipBP} together with \cite[Proposition 9.1.3]{Cohen07v2} (see also \cite[formula 24.4.3]{DLMF})
implies that
\begin{equation}\label{Bz1z}
{\cal B}_\alpha(-\frac{z}{1-z}) = \big(-\frac{z}{1-z}\big)^\alpha  {\tilde {\cal B}}_\alpha(1-\frac1z) = \big(\frac{z}{1-z}\big)^\alpha {\tilde {\cal B}}_\alpha(\frac1z) = (1-z)^{-\alpha} {\cal B}_\alpha(z).
\end{equation}

We use Lagrange inversion~\cite[formula 2.1.8]{Gessel2016},
which for any power series $H(z)$ and $\phi(z)$ with $\phi(0)\neq 0$, and any integer $m$ states: 
\begin{equation}\label{lb}
\coeff{z^m} H(z)\phi(z)^m = \coeff{z^m} H(g(z))\frac{g'(z)z}{g(z)},
\end{equation}
where $g(z)$ is the compositional inverse of $\frac{z}{\phi(z)}$.
It implies the theorem statement by taking $\phi(z) := 1+z$ and $H(z) := h(\frac{z}{1+z})(1+z)^{-1-\alpha} {\cal B}_\alpha(-z)$, giving $g(z) = \frac{z}{1-z}$ and then by \eqref{Bz1z}
$$H(g(z))\frac{g'(z)z}{g(z)} = h(z)(1-z)^{1+\alpha} {\cal B}_\alpha(-\frac{z}{1-z})\frac1{1-z}=h(z){\cal B}_\alpha(z).$$
\end{proof}

We will find it useful to express the coefficients $A^k_i$ and polynomials $P_k(n)$ in terms of Bernoulli series as follows:

\begin{theorem}\label{th:gf_A_1}
For odd integers $k$ and $i$ with $1\leq i\leq k+2$, we have
$$
A^k_i = \frac1{(\frac{k}2+1)2^{k+1}} \coeff{z^{\frac{k-i}2+1}} (1+z)^{k+2} {\cal B}_{\frac{k}2+1}(\frac{4z}{(1+z)^2})
$$
and
$$
P_k(n) = \frac{1}{\frac{k}2+1} \coeff{z^\frac{k+1}2} (1-nz)^{-1} {\cal B}_{\frac{k}{2}+1}(-z) .
$$
\end{theorem}

\begin{proof}
Substituting $j\mapsto \frac{k-i}{2}+1-j$ in \eqref{Aki}, we obtain the formula for $A^k_i$ as follows:
\[
\begin{split}
A^k_i &= \frac{1}{(\frac{k}{2}+1)2^{i-1}}\sum_{j=0}^{\frac{k-i}{2}+1}\binom{k+2-2j}{\frac{k-i}{2}+1-j}\frac{1}{4^{\frac{k-i}{2}+1-j}}\binom{\frac{k}{2}+1}{j}B_{j} \\
&= \frac{1}{(\frac{k}{2}+1)2^{k+1}}\sum_{j=0}^{\frac{k-i}{2}+1} 4^j\binom{\frac{k}{2}+1}{j}B_{j} \coeff{z^{\frac{k-i}{2}+1-j}} (1+z)^{k+2-2j} \\
&= \frac{1}{(\frac{k}{2}+1)2^{k+1}} \coeff{z^{\frac{k-i}{2}+1}}(1+z)^{k+2}\sum_{j=0}^{\frac{k-i}{2}+1} 4^j\binom{\frac{k}{2}+1}{j}B_{j} z^j (1+z)^{-2j} \\
&= \frac{1}{(\frac{k}{2}+1)2^{k+1}} \coeff{z^{\frac{k-i}{2}+1}} (1+z)^{k+2} {\cal B}_{\frac{k}2+1}(\frac{4z}{(1+z)^2}).
\end{split}
\]

The formula for $P_k(n)$ is derived from \eqref{Pkn} as follows:
\[
\begin{split}
P_k(n) &=  \frac1{\frac{k}2+1}\sum_{i=0}^{\frac{k+1}{2}} \binom{\frac{k}{2}+1}{i}n^{\frac{k+1}{2}-i} (-1)^i B_i \\
&= \frac1{\frac{k}2+1}\sum_{i=0}^{\frac{k+1}{2}} \binom{\frac{k}{2}+1}{i} (-1)^i B_i \coeff{z^{\frac{k+1}{2}-i}} (1-nz)^{-1} \\
&= \frac1{\frac{k}2+1} \coeff{z^{\frac{k+1}{2}}} (1-nz)^{-1}\sum_{i=0}^{\frac{k+1}{2}} \binom{\frac{k}{2}+1}{i} (-z)^i B_i  \\
&= \frac1{\frac{k}2+1} \coeff{z^{\frac{k+1}{2}}} (1-nz)^{-1} {\cal B}_{\frac{k}{2}+1}(-z).
\end{split}
\]
\end{proof}

We will be using the properties of the generating function for Catalan numbers~\cite[Section 5.4]{Graham}: 
\begin{equation}\label{eq:catalan}
{\cal C}(z) := \frac{1-(1-4z)^{1/2}}{2z},
\end{equation}
which we summarize in the following lemma.

\begin{lemma}\label{lem:catalan}
    The following identities hold:
    \begin{enumerate}[\normalfont(i)]
        \item $\binom{2j+i}{j}\frac{1}{4^j}=\coeff{z^j} {\cal C}(\frac{z}4)^i(1-z)^{-1/2}$;
        \item $\cal C(\frac{z}4)=2\frac{1-(1-z)^{1/2}}z$;
        \item $\cal C(\frac{z}{4(1+z)})=\cal C(-\frac{z}4)(1+z)^{1/2}$; % and $\cal C(-\frac{z}{4(1-z)})=\cal C(\frac{z}4)(1-z)^{1/2}$;
        \item $\cal C(\frac{z}4)(1-z)^{-1/2} = 2\frac{(1-z)^{-1/2}-1}z = \frac{2-\cal C(\frac{z}4)}{1-z}$;
        \item $\cal C(\frac{z}4)^2\frac{z}4 = \cal C(\frac{z}4)-1$.
    \end{enumerate}
\end{lemma}

\begin{proof}
The identity (i) can be seen in \cite[formula 5.72]{Graham}.
The others can be easily verified using the explicit formula~\eqref{eq:catalan} and/or the functional identity $z\cal C(z)^2 - \cal C(z) + 1 = 0$.
\end{proof}

In the following theorem, we express the coefficients $A^k_i$ in terms of $\cal C(z)$ and ${\cal B}_{\frac{k}2+1}(z)$, and establish cases where they are zero.

\begin{theorem}\label{th:gf_A_2}
For odd integers $k$ and $i$ with $1\leq i\leq k+2$, we have
$$
A^k_i = \frac1{(\frac{k}2+1)2^{i-1}} \coeff{z^{\frac{k-i}2+1}} {\cal C}(\frac{z}4)^i(1-z)^{-1/2}{\cal B}_{\frac{k}2+1}(z).
$$
Furthermore, $A^k_i=0$ whenever $i=1$ or $\frac{k-i}{2}$ is even.
\end{theorem}

\begin{proof}
Using Lemma~\ref{lem:catalan}(i) and noting that $\binom{\frac{k}{2}+1}{\frac{k-i}{2}+1-j}B_{\frac{k-i}{2}+1-j}=\coeff{z^{\frac{k-i}{2}+1-j}} {\cal B}_{\frac{k}{2}+1}(z)$, we rewrite the formula~\eqref{Aki} as
\begin{equation}\label{eq:Aki_1}
A^k_i = \frac1{(\frac{k}2+1)2^{i-1}} \coeff{z^{\frac{k-i}2+1}} {\cal C}(\frac{z}4)^i(1-z)^{-1/2}{\cal B}_{\frac{k}2+1}(z).
\end{equation}

Applying Theorem~\ref{th:B_duality} for $m=\frac{k-i}2+1$ and $h(z)={\cal C}(\frac{z}4)^i(1-z)^{-1/2}$, we obtain another formula:
\[
\begin{split}
A^k_i &= \frac1{(\frac{k}2+1)2^{i-1}} \coeff{z^{\frac{k-i}2+1}} {\cal C}(\frac{z}{4(1+z)})^i(1+z)^{-1/2-i/2}{\cal B}_{\frac{k}2+1}(-z) \\
&=
\frac{(-1)^{\frac{k-i}2+1}}{(\frac{k}2+1)2^{i-1}} \coeff{z^{\frac{k-i}2+1}} {\cal C}(-\frac{z}{4(1-z)})^i(1-z)^{-1/2-i/2}{\cal B}_{\frac{k}2+1}(z)
\end{split}
\]
and by Lemma~\ref{lem:catalan}(iii),
\begin{equation}\label{eq:Aki_2}
A^k_i = \frac{(-1)^{\frac{k-i}2+1}}{(\frac{k}2+1)2^{i-1}} \coeff{z^{\frac{k-i}2+1}} {\cal C}(\frac{z}4)^i(1-z)^{-1/2}{\cal B}_{\frac{k}2+1}(z).
\end{equation}
Comparing the formulae \eqref{eq:Aki_1} and \eqref{eq:Aki_2}, we conclude that $A^k_i = 0$ whenever $\frac{k-i}2 + 1$ is odd.

It remains to prove that $A^k_1=0$ when $\frac{k+1}2$ is even. Using Lemma~\ref{lem:catalan}(iv) and noticing that $\coeff{z^{\frac{k+3}2}} {\cal B}_{\frac{k}2+1}(z)=0$, we simplify the formula~\eqref{eq:Aki_1} for $i=1$ as follows:
$$A^k_1 = \frac2{\frac{k}2+1} \coeff{z^{\frac{k+1}2}} \frac{(1-z)^{-1/2}-1}z{\cal B}_{\frac{k}2+1}(z) = \frac2{\frac{k}2+1} \coeff{z^{\frac{k+3}2}} (1-z)^{-1/2} {\cal B}_{\frac{k}2+1}(z).$$
From Theorem~\ref{th:B_duality} for $m=\frac{k+3}2$ and $h(z)=(1-z)^{-1/2}$, we have
$$
\coeff{z^{\frac{k+3}2}} (1-z)^{-1/2} {\cal B}_{\frac{k}2+1}(z) = \coeff{z^{\frac{k+3}2}} {\cal B}_{\frac{k}2+1}(-z) = 0,
$$
implying that $A^k_1=0$.
\end{proof}

\begin{theorem}\label{th:sumA_gf}
For an odd positive integer $k$ and any power series $F(z)$, the following two formulae hold:
\begin{align}
\sum_{i=1\atop i\text{ is odd}}^{k+2} A^k_i\cdot \coeff{z^{\frac{i-1}2}} F(z) &= 
\frac1{\frac{k}2+1} \coeff{z^{\frac{k+1}2}} {\cal B}_{\frac{k}2+1}(z) \frac{2-{\cal C}(\frac{z}4)}{1-z} F({\cal C}(\frac{z}4)-1) \tag{i}\\
&=
\frac1{\frac{k}2+1} \coeff{z^{\frac{k+1}2}} {\cal B}_{\frac{k}2+1}(-z) \frac{2-{\cal C}(-\frac{z}4)}{1+z} F(1-{\cal C}(-\frac{z}4)). \tag{ii}
\end{align}
\end{theorem}

\begin{proof}
The formula (i) can be derived from Theorem~\ref{th:gf_A_2} and Lemma~\ref{lem:catalan}(iv,v) as follows:
\[
\begin{split}
\sum_{i=1\atop i\text{ is odd}}^{k+2} A^k_i\cdot \coeff{t^{\frac{i-1}2}} F(t) &=
\frac1{\frac{k}2+1} 
\sum_{i=1\atop i\text{ is odd}}^{k+2} \coeff{z^{\frac{k-i}2+1}} \frac1{2^{i-1}}{\cal C}(\frac{z}4)^i(1-z)^{-1/2}{\cal B}_{\frac{k}2+1}(z)\cdot \coeff{t^{\frac{i-1}2}} F(t) \\
&=
\frac1{\frac{k}2+1} \coeff{z^{\frac{k+1}2}} {\cal B}_{\frac{k}2+1}(z) (1-z)^{-1/2} {\cal C}(\frac{z}4)
\sum_{i=1\atop i\text{ is odd}}^{k+2} \bigg({\cal C}(\frac{z}4)^2\frac{z}4\bigg)^{\frac{i-1}2}\cdot \coeff{t^{\frac{i-1}2}} F(t) \\
&=
\frac1{\frac{k}2+1} \coeff{z^{\frac{k+1}2}} {\cal B}_{\frac{k}2+1}(z) \frac{2-{\cal C}(\frac{z}4)}{1-z}
 F({\cal C}(\frac{z}4)-1).
\end{split}
\]

Now, the formula (ii) follows from (i) by applying Theorem~\ref{th:B_duality} and Lemma~\ref{lem:catalan}(iii):
$$\coeff{z^{\frac{k+1}2}} {\cal B}_{\frac{k}2+1}(z) \frac{2-{\cal C}(\frac{z}4)}{1-z} F({\cal C}(\frac{z}4)-1)
= \coeff{z^{\frac{k+1}2}} {\cal B}_{\frac{k}2+1}(-z) \frac{2-{\cal C}(-\frac{z}4)}{1+z} F(1-{\cal C}(-\frac{z}4)).
$$
\end{proof}

\section{Proof of Theorem~\ref{th:main}}

We start by expressing $(\sqrt{n}-\sqrt{n+1})^m$ for an odd positive
integer $m$ in terms of Chebyshev polynomials of the first kind $T_k$ and the second kind $U_k$~\cite[Section 18.5]{DLMF}.

\begin{lemma}\label{lem:sqrt_power_m}
For an odd positive integer $m$,
\[
(\sqrt{n}-\sqrt{n+1})^m = \sqrt{n}\cdot U_{m-1}(\sqrt{n+1}) - \sqrt{n+1}\cdot \frac{T_m(\sqrt{n+1})}{\sqrt{n+1}},
\]
where $U_{m-1}(\sqrt{n+1})$ and $\frac{T_m(\sqrt{n+1})}{\sqrt{n+1}}$ represent polynomials in $n$, with the following generating functions:
\[
\begin{split}
{\cal T}_n(z) &:= \sum_{m=1\atop m \text{is odd}}^\infty T_m(\sqrt{n+1}) z^{\frac{m-1}2} = \sqrt{n+1} \frac{1 - z}{1 - 2(2n+1)z+z^2},\\
{\cal U}_n(z) &:= \sum_{m=1\atop m \text{is odd}}^\infty U_{m-1}(\sqrt{n+1}) z^{\frac{m-1}2}
=\frac{1+z}{1 - 2(2n+1)z + z^2}.
\end{split}
\]
\end{lemma}

\begin{proof}
Let us introduce a parameter $\beta$ such that $\sqrt{n}=\sinh(\beta)$. Then $\sqrt{n+1}=\cosh(\beta)$ and
$$
(\sqrt{n}-\sqrt{n+1})^m = (\sinh(\beta) - \cosh(\beta))^m = \sinh(m\beta) - \cosh(m\beta) 
= \sinh(\beta) U_{m-1}(\cosh(\beta)) - T_m(\cosh(\beta)).
$$
It remains to note that $\frac{T_m(z)}z$ and $U_{m-1}(z)$ are polynomials in $z^2$. The generating functions for $T_m(\sqrt{n+1})$ and $U_{m-1}(\sqrt{n})$ are easily obtained from the odd parts of the generating functions for Chebyshev polynomials (e.g., see \cite[formulae 18.12.8	and 18.12.10]{DLMF}).
\end{proof}

Generalizing Ramanujan's function $\phi_1(n)$ given by \eqref{phi1} for an odd positive integer $k$, we define
\begin{equation}
\phi_{k}(n) := \sum_{i=1}^n i^{\frac{k}{2}} - C_k - \sqrt{n}\,P_k(n) - \sum_{i=3\atop i\text{ is odd}}^{k+2} A^k_i\,\tau(n,i).
\end{equation}
The definition \eqref{Ck} implies that $\phi_k(1) = 0$. Then, using Lemma~\ref{lem:sqrt_power_m} and $A^k_1=0$ (Theorem~\ref{th:gf_A_2}), we have
\begin{align*}
   &\phi_k(n)-\phi_k(n+1) = -(n+1)^{\frac{k}{2}} - \sqrt{n}P_k(n) + \sqrt{n+1}P_k(n+1)+\sum_{i=1\atop i\text{ is odd}}^{k+2} A^k_i(\sqrt{n}-\sqrt{n+1})^i\\ 
   &~=-(n+1)^{\frac{k}{2}} - \sqrt{n}P_k(n) + \sqrt{n+1}P_k(n+1)+\sqrt{n}\sum_{i=1\atop i\text{ is odd}}^{k+2} A^k_i\,\cdot U_{i-1}(\sqrt{n+1}) - \sum_{i=1\atop i\text{ is odd}}^{k+2} A^k_i\,T_i(\sqrt{n+1}). 
\end{align*}

As proved in Lemma~\ref{lem:sumATU} below, the last expression is identically zero, from which Theorem~\ref{th:main} follows instantly.

\begin{lemma}\label{lem:sumATU}
For any odd positive integer $k$, 
$$
\sum_{i=1\atop i\text{ is odd}}^{k+2} A^k_i\,T_i(\sqrt{n+1}) = \sqrt{n+1}P_k(n+1) -(n+1)^{\frac{k}{2}} 
$$  
and
$$
\sum_{i=1\atop i\text{ is odd}}^{k+2} A^k_i\,\cdot U_{i-1}(\sqrt{n+1}) = P_k(n).
$$
\end{lemma}

\begin{proof}
Using the expression for ${\cal T}_n$ given in Lemma~\ref{lem:sqrt_power_m}, it is easy to verify that
$$\frac{2-{\cal C}(\frac{z}4)}{1-z} {\cal T}_n({\cal C}(\frac{z}4)-1) = \sqrt{n+1}\ (1-(n+1)z)^{-1},$$
and thus Theorem~\ref{th:sumA_gf}(i) for $F = {\cal T}_n$, formula~\eqref{eq:B_minus_z}, and Theorem~\ref{th:gf_A_1} imply that
\[
\begin{split}
\sum_{i=1\atop i\text{ is odd}}^{k+2} A^k_i\,T_i(\sqrt{n+1}) &=
\frac{\sqrt{n+1}}{\frac{k}2+1} \coeff{z^{\frac{k+1}2}} {\cal B}_{\frac{k}2+1}(z) (1-(n+1)z)^{-1} \\
&= \frac{\sqrt{n+1}}{\frac{k}2+1} \coeff{z^{\frac{k+1}2}} \big({\cal B}_{\frac{k}2+1}(-z) - (\frac{k}2+1)z\big) (1-(n+1)z)^{-1} \\
&= {\sqrt{n+1}}P_k(n+1) - (n+1)^{\frac{k}2}.
\end{split}
\]

For the expression of ${\cal U}_n$ in Lemma~\ref{lem:sqrt_power_m}, we can similarly verify that
$$\frac{2-{\cal C}(-\frac{z}4)}{1+z} {\cal U}_n(1-{\cal C}(-\frac{z}4)) = (1-nz)^{-1},$$
and then Theorem~\ref{th:sumA_gf}(ii) for $F = {\cal U}_n$ and Theorem~\ref{th:gf_A_1} imply that
$$
\sum_{i=1\atop i\text{ is odd}}^{k+2} A^k_i\,U_{i-1}(\sqrt{n+1}) =
\frac1{\frac{k}2+1} \coeff{z^{\frac{k+1}2}} {\cal B}_{\frac{k}2+1}(-z) (1-nz)^{-1} = P_k(n).
$$
\end{proof}

\section*{Acknowledgments} The last four authors were partly supported by the National Science Foundation grant 1820731. The last author was additionally supported by the Professional Staff Congress-City University of New York grant \emph{`On a generalization of a result of Ramanujan on the sum of certain roots of natural numbers'}.

\bibliographystyle{plainurl}
\bibliography{refs}

\begin{thebibliography}{1}

\bibitem{Cohen07v2}
H.~Cohen.
\newblock {\em {Number Theory, Volume II: Analytic and Modern Tools}}.
\newblock Springer Science + Business Media, LLC, New York, NY, 2007.
\newblock \href {https://doi.org/10.1007/978-0-387-49894-2}
  {\path{doi:10.1007/978-0-387-49894-2}}.

\bibitem{Gessel2016}
I.~M. Gessel.
\newblock Lagrange inversion.
\newblock {\em Journal of Combinatorial Theory, Series A}, 144:212--249, 2016.
\newblock \href {https://doi.org/10.1016/j.jcta.2016.06.018}
  {\path{doi:10.1016/j.jcta.2016.06.018}}.

\bibitem{Graham}
R.~L. Graham, D.~E. Knuth, and O.~Patashnik.
\newblock {\em Concrete mathematics: a foundation for computer science}.
\newblock Advanced Book Program. Addison-Wesley Publishing Company, Reading,
  MA, USA, 1988.

\bibitem{Kellner2023}
B.~C. Kellner.
\newblock Faulhaber polynomials and reciprocal {B}ernoulli polynomials.
\newblock {\em Rocky Mountain Journal of Mathematics}, 53(1):119--151, 2023.
\newblock \href {https://doi.org/10.1216/rmj.2023.53.119}
  {\path{doi:10.1216/rmj.2023.53.119}}.

\bibitem{knuth1993}
D.~E. Knuth.
\newblock {Johann Faulhaber and sums of powers}.
\newblock {\em Math. Comp.}, 61:277--294, 1993.
\newblock \href {https://doi.org/10.1090/S0025-5718-1993-1197512-7}
  {\path{doi:10.1090/S0025-5718-1993-1197512-7}}.

\bibitem{gown}
K.~J. McGown and H.~R. Parks.
\newblock {The generalization of Faulhaber's formula to sums of non-integral
  powers}.
\newblock {\em Journal of Mathematical Analysis and Applications},
  330(1):571--575, 2007.
\newblock \href {https://doi.org/10.1016/j.jmaa.2006.08.019}
  {\path{doi:10.1016/j.jmaa.2006.08.019}}.

\bibitem{DLMF}
F.~W.~J. Olver, A.~B. {Olde Daalhuis}, D.~W. Lozier, B.~I. Schneider, R.~F.
  Boisvert, C.~W. Clark, B.~R. Miller, B.~V. Saunders, H.~S. Cohl, and
  M.~A.~McClain (eds.).
\newblock {\it NIST Digital Library of Mathematical Functions}.
\newblock \url{https://dlmf.nist.gov/}, Release 1.2.4 of 2025-03-15.
\newblock Access date: 2025-09-24.

\bibitem{ram}
S.~Ramanujan.
\newblock On the sum of the square roots of the first $n$ natural numbers.
\newblock {\em J. Indian Math. Soc}, 7:173--175, 1915.

\bibitem{shek}
S.~Shekatkar.
\newblock The sum of the $r$'th roots of first $n$ natural numbers and new
  formula for factorial.
\newblock {\em Preprint arXiv:1204.0877 [math.NT]}, 2013.
\newblock \href {https://doi.org/10.48550/arXiv.1204.0877}
  {\path{doi:10.48550/arXiv.1204.0877}}.

\end{thebibliography}

\end{document}